\documentclass{daj}

\usepackage{latexsym,amssymb,amsfonts,amsmath,mathrsfs,float, amsthm}
\usepackage{xcolor}

  \newtheorem{theorem}{Theorem}[section]
  \newtheorem{lemma}[theorem]{Lemma}
  \newtheorem{proposition}[theorem]{Proposition}
  
  \newtheorem{corollary}[theorem]{Corollary}

  \newtheorem{definition}[theorem]{Definition}

  \newtheorem{remark}[theorem]{Remark}
  
  \renewenvironment{proof}[1][]{
    	\begin{trivlist}
     	\item[\hspace{\labelsep}{\em\noindent Proof#1:\/}]}
     	{{\hfill$\Box$}
    	\end{trivlist}
  }

  \renewcommand{\Pr}{\mbox{\rm Pr}}
  \newcommand{\Exp}{{\mathbb{E}}}

  \newcommand{\R}{\mathbb{R}} 
  \newcommand{\C}{\mathbb{C}} 
  \newcommand{\N}{\mathbb{N}} 
  \newcommand{\F}{\mathbb{F}} 
  \newcommand{\bset}[1]{\{0,1\}^{#1}} 
  \newcommand{\1}{\mathbf{1}}

  \DeclareMathOperator{\charac}{char}
  \newcommand{\one}{\mathbf{1}}

  \newcommand{\st}{:\,} 

  \newcommand{\eps}{\varepsilon}

  \DeclareMathOperator{\rank}{rk}

  \DeclareMathOperator{\arank}{arank}
  \DeclareMathOperator{\prank}{prank} 
  \DeclareMathOperator{\srank}{srank} 
  \DeclareMathOperator{\GR}{GR} 
  \DeclareMathOperator{\bias}{bias}
  \DeclareMathOperator{\Pol}{Pol}
  \DeclareMathOperator{\codim}{codim}

  \newcommand{\beq}{\begin{equation}}
  \newcommand{\eeq}{\end{equation}}
  \newcommand{\beqn}{\begin{equation*}}
  \newcommand{\eeqn}{\end{equation*}}
  \newcommand{\beqr}{\begin{eqnarray}}
  \newcommand{\eeqr}{\end{eqnarray}}
  \newcommand{\beqrn}{\begin{eqnarray*}}
  \newcommand{\eeqrn}{\end{eqnarray*}}
  \newcommand{\bmline}{\begin{multline}}
  \newcommand{\emline}{\end{multline}}
  \newcommand{\bmlinen}{\begin{multline*}}
  \newcommand{\emlinen}{\end{multline*}}
  
\dajAUTHORdetails{%
  title = {Random restrictions of high-rank tensors and polynomial maps}, 
  author = {Jop Bri\"{e}t and Davi Castro-Silva},
  plaintextauthor = {Jop Briët, Davi Castro-Silva},
  runningtitle = {Random restrictions of tensors and polynomials}, 
  keywords = {slice rank, analytic rank, Schmidt rank, Talagrand inequality, analysis of boolean functions},
}   

\dajEDITORdetails{%
   year={2024},
   number={9},
   received={23 February 2023},   
   published={4 November 2024},  
   doi={10.19086/da.124610},       
}   

\begin{document}

\begin{frontmatter}[classification=text]

\title{Random Restrictions of High-Rank Tensors and Polynomial Maps} 

\author[jb]{Jop Bri\"{e}t\thanks{Supported by the Dutch Research Council (NWO) as part of the NETWORKS programme (grant no. 024.002.003).}}
\author[dcs]{Davi Castro-Silva\thanks{Supported by the Dutch Research Council (NWO) as part of the NETWORKS programme (grant no. 024.002.003).}}

\begin{abstract}
Motivated by a problem in computational complexity, we consider the behavior of rank functions for tensors and polynomial maps under random coordinate restrictions.
We show that, for a broad class of rank functions called \emph{natural rank functions}, random coordinate restriction to a dense set will typically reduce the rank by at most a constant factor.
\end{abstract}
\end{frontmatter}


%
\section{Introduction}
\label{sec:intro}

Different but equivalent definitions of matrix rank have been generalized to truly different  rank functions for tensors.
Although they have proved useful in a variety of applications, the basic theory of these rank functions, describing for instance their interrelations and elementary properties, is still far from complete.

Without going into the definitions (which are given in Section~\ref{sec:tensors}), we mention a number of these rank functions to indicate some of the contexts in which they have appeared.
The \emph{slice rank} of a tensor was introduced by Tao~\cite{Tao:srank, TaoSawin} to reformulate the breakthrough proof of the cap set conjecture due to Croot, Lev and Pach~\cite{CrootLP2017} and Ellenberg and Gijswijt~\cite{EllenbergG2017}.
Slice rank is generalized by the \emph{partition rank}, which was introduced by Naslund to prove bounds on the size of subsets of~$\F_q^n$ without $k$-right corners~\cite{Naslund2020}, as well as provide exponential improvements on the Erd\H{o}s--Ginzburg--Ziv constant~\cite{NaslundEGZ}.
The \emph{analytic rank} is based on a measure of equidistribution for multilinear forms associated to tensors over finite fields, and was introduced by Gowers and Wolf to study solutions to linear systems of equations in large subsets of finite vectors spaces~\cite{GowersW2011}.
\emph{Geometric rank}, defined and studied by Kopparty, Moshkovitz and Zuiddam in the context of algebraic complexity theory~\cite{KoppartyMZ}, gives a natural analogue of analytic rank for tensors over infinite fields.

Closely related to these rank functions for tensors  are notions of rank for multivariate polynomials.
As quadratic forms on finite-dimensional vector spaces are naturally represented by matrices after choosing a basis, matrix rank gives a corresponding notion of rank for quadratic forms.
This, in turn, may be generalized to rank functions for arbitrary polynomials by considering their associated homogeneous multilinear forms.
Specific problems concerning multivariate polynomials might also give rise to other notions of rank more well-suited to the application at hand.

A notion of polynomial rank akin to the partition rank of tensors was used already in the '80s by Schmidt in work on number theory~\cite{Schmidt:1985}.
This notion has been re-discovered and proven useful on several occasions (see Section~\ref{sec:polynomials}), and is referred to by several different names in the literature;
here we will refer to it as \emph{Schmidt rank}.
Work on the Inverse Theorem for the Gowers uniformity norms led Green and Tao to define the notion of \emph{degree rank}~\cite{GreenT2009}, which quantifies how hard it is to express the considered polynomial as a function of lower-degree polynomials;
this notion was shown to be closely linked to equidistribution properties of multivariate polynomials over prime fields $\F_p$.
Tao and Ziegler~\cite{TaoZ12} later studied the relationship between the degree rank of a polynomial\footnote{The notion of degree rank studied by Tao and Ziegler was actually a slight modification of the original one, where they also allow for ``non-classical polynomials'' when expressing higher-degree polynomials in terms of lower-degree ones.}
and its \emph{analytic rank}, defined as the (tensor) analytic rank of its associated homogeneous multilinear form, and exploited their close connection in order to prove the general case of the Gowers Inverse Theorem over~$\F_p^n$.

Recent work on constant-depth Boolean circuits by Buhrman, Neumann and the present authors gave rise to a problem on equidistribution properties of higher-dimensional polynomial maps under biased input distributions~\cite{BrietBCSN}.
This motivated a new notion of analytic rank for (high-dimensional) polynomial maps and prompted the study of rank under random coordinate restrictions, which is the topic of this paper.

\medskip
Common to the tensors, polynomials and polynomial maps considered here is that they can be viewed as maps on~$\F^X$, where~$\F$ is a given field and $X$ is a finite set indexing the variables.
The main question addressed in this paper is whether, if a map~$\phi$ on~$\F^X$ has high rank, most of its coordinate restrictions~$\phi_{|I}$ on~$\F^I$ also have high rank for dense subsets~$I\subseteq X$
(where we also respect the product structure of~$X$ in the case of tensors).
Our main results show that this is the case for all ``natural'' rank functions, which include all those mentioned above.

\subsection{The matrix case}
\label{sec:matrixcase}

It is instructive to first consider the case of matrices, which is simpler and illustrates the spirit of our main results.
For a matrix $A\in \F^{n\times n}$ and subsets $I,J\subseteq [n]$, denote by~$A_{|I\times J}$ the sub-matrix of~$A$ induced by the rows in~$I$ and columns in~$J$.
Given $\sigma\in (0,1)$, consider a random set $I\subseteq [n]$ containing each element independently with probability~$\sigma$;
we write $I\sim[n]_\sigma$ when~$I$ is distributed as such.
Note that, if $I\sim[n]_\rho$ and $J\sim[n]_\sigma$ are independent, then $I\cup J\sim [n]_\eta$ with $\eta = 1 - (1 - \rho)(1 - \sigma)$.

\begin{proposition}\label{prop:matrixcase}
For every $\sigma\in (0,1]$ there exists $\kappa\in (0,1]$ such that for every matrix $A\in \F^{n\times n}$ we have
\beqn
\Pr_{I\sim[n]_\sigma}\big[\rank(A_{|I\times I}) \geq \kappa \cdot \rank(A)\big] \geq 1 - 2e^{-\kappa \rank(A)}.
\eeqn
\end{proposition}

\begin{proof}
Write $\rho=1 - \sqrt{1 - \sigma}$ and let $J,J'\sim[n]_\rho$ be independent random sets;
note that $J\cup J'\sim[n]_\sigma$.
Let $r = \rank(A)$, and fix a set $S\subseteq [n]$ of~$r$ linearly independent rows of~$A$.
By the Chernoff bound~\cite{HagerupRub:1990}, the probability that the set~$J$ satisfies $|J\cap S| < \rho r/2$ is at most~$e^{-\rho r/8}$.

Now let~$B := A_{|(J\cap S) \times [n]}$ be the (random) sub-matrix of~$A$ formed by the rows in~$J\cap S$.
Since its rows are linearly independent, the rank of $B$ is precisely~$|J\cap S|$;
let $T\subseteq [n]$ be a set of $|J\cap S|$ linearly independent columns of~$B$.
Then the probability that $|J'\cap T| < \rho|T|/2$ is at most~$e^{-\rho |T|/8}$, and the rank of $B_{|(J\cap S) \times (J'\cap T)} = A_{|(J\cap S) \times (J'\cap T)}$ is equal to $|J'\cap T|$.
It follows from the union bound and monotonicity of rank under restrictions that, with probability at least $1 - 2e^{-\rho^2 r/16}$, the principal sub-matrix of~$A$ induced by $J\cup J'$ has rank at least $\rho^2 r/4$.
The result now follows since $J\cup J'\sim[n]_\sigma$.
\end{proof}

\subsection{Main results and outline of the paper}

Here we generalize Proposition~\ref{prop:matrixcase} to tensors and polynomial maps for rank functions that satisfy a few natural properties, namely
``sub-additivity'',
``monotonicity'',
a ``Lipschitz condition'' and, in the case of polynomial maps,
``symmetry''
(see Section~\ref{sec:tensors} and Section~\ref{sec:polynomials} for the precise definitions).
Those functions which satisfy these properties are called \emph{natural rank functions};
we note that all notions of rank mentioned above are natural rank functions.

Since our results are independent of the field considered (which can be finite or infinite), we will always denote it by $\F$ and suppress statements of the form ``let $\F$ be a field'' or ``for every field $\F$''.
We begin by considering the case of tensors.

\begin{definition}[Tensors]\label{def:tensors}
For finite sets $X_1,\dots,X_d \subset \N$, a $d$-tensor is a map $T:X_1\times\cdots\times X_d\to \F$.
We will associate with any $d$-tensor~$T$ a multilinear map $\F^{X_1}\times\cdots\times \F^{X_d}\to \F$ and an element of $\F^{X_1}\otimes \cdots\otimes \F^{X_d}$ in the obvious way, and also denote these objects by~$T$.
\end{definition}

\begin{definition}[Restriction of tensors]
For a tensor~$T$ as in Definition~\ref{def:tensors} and subsets $I_1\subseteq X_1,\dots,I_d\subseteq X_d$, denote $I_{[d]} = I_1\times \cdots\times I_d$ and write~$T_{|I_{[d]}}$ for the restriction of~$T$ to~$I_{[d]}$.
If~$T$ is viewed as an element of $\F^{X_1}\otimes \cdots\otimes \F^{X_d}$, then~$T_{|I_{[d]}}$ is simply a sub-tensor.
\end{definition}

We define $(\F^\infty)^{\otimes d}$ to be the set of $d$-tensors over~$\F$ with finite support.
Note that the tensors defined on finite sets naturally embed into this set, and that the rank functions for tensors discussed above are invariant under this embedding.
Our main result regarding tensors is then as follows:

\begin{theorem}\label{thm:random_tensor}
For every~$d\in \N$ and~$\sigma\in (0,1]$, there exist constants $C, \kappa > 0$ such that the following holds.
For every natural rank function $\rank: (\F^{\infty})^{\otimes d} \to \R_+$ and every $d$-tensor $T \in \bigotimes_{i=1}^d \F^{[n_i]}$ we have
\beqn
\Pr_{I_1\sim [n_1]_{\sigma}, \dots, I_d\sim [n_d]_{\sigma}}\big[ \rank \big(T_{|I_{[d]}}\big) \geq \kappa\cdot \rank(T)\big] \geq 1 - C e^{-\kappa \rank(T)}.
\eeqn
\end{theorem}

As noted before, the union of independent Bernoulli-random subsets of~$[n]$ is again a Bernoulli-random subset.
As a consequence of this and monotonicity under restrictions, in the standard case of ``cubic'' tensors where every row is indexed by the same set, one also obtains the following symmetric version of the last theorem:

\begin{corollary}
For every~$d\in \N$ and~$\sigma\in (0,1]$, there exist constants $C, \kappa > 0$ such that the following holds.
For every natural rank function $\rank: (\F^{\infty})^{\otimes d} \to \R_+$ and every $d$-tensor $T \in (\F^n)^{\otimes d}$ we have
\beqn
\Pr_{I\sim [n]_{\sigma}}\big[ \rank \big(T_{|I^d}\big) \geq \kappa\cdot \rank(T)\big] \geq 1 - C e^{-\kappa \rank(T)}.
\eeqn
\end{corollary}

Whereas the proof of the matrix case (Proposition~\ref{prop:matrixcase}) uses the fact that a rank-$r$ matrix contains a full-rank $r\times r$ submatrix, the proof of the general case of Theorem~\ref{thm:random_tensor} proceeds differently and instead uses ideas from probability theory, in particular concerning concentration inequalities on product spaces.
It will be presented in Section~\ref{sec:tensors}.

Recently, the topic of high-rank restrictions of tensors was explored in detail by Karam~\cite{Karam2022}, and Gowers provided an example of a 3-tensor of slice rank $4$ which does not have a full-rank $4\times 4\times 4$ subtensor
(see \cite[Proposition 3.1]{Karam2022}).
There are some interesting parallels between our results and those of Karam, which will be discussed later on in Section~\ref{sec:discussion};
to the best of our knowledge, his results do not suffice to establish our main theorem above, nor the other way around.

Next we consider the setting of polynomial maps, which are formally defined as follows:

\begin{definition}[Polynomial map]\label{def:polymap}
A polynomial map is an ordered tuple $\phi(x) = \big(f_1(x),\dots,f_k(x)\big)$ of polynomials $f_1,\dots,f_k\in \F[x_1,\dots,x_n]$.
We identify with~$\phi$ a map $\F^n\to\F^k$ in the natural way.
The degree of~$\phi$ is the maximum degree of the~$f_i$.
\end{definition}

\begin{definition}[Restriction of polynomial maps]
For a polynomial map~$\phi:\F^n\to\F^k$ and a set~$I\subseteq [n]$, define the restriction~$\phi_{|I}:\F^I\to\F^k$ to be the map given by
$\phi_{|I}(y) = \phi(\bar y)$,
where~$\bar y\in \F^n$ agrees with~$y$ on the coordinates in~$I$ and is zero elsewhere.
\end{definition}

We denote the space of all polynomial maps $\phi: \F^n\to\F^k$ of degree at most~$d$ by $\Pol_{\leq d}(\F^n, \F^k)$, and write
\beqn
\Pol_{\leq d}(\F^{\infty}, \F^k) = \bigcup_{n\in \N} \Pol_{\leq d}(\F^n, \F^k).
\eeqn
Our main result in this setting is the following:

\begin{theorem} \label{thm:random_poly}
For every~$d\in \N$ and~$\sigma,\eps\in (0,1]$, there exist constants~$\kappa = \kappa(d, \sigma)>0$ and~$R = R(d,\sigma,\eps)\in\N$ such that the following holds.
For every natural rank function~$\rank: \Pol_{\leq d}(\F^{\infty}, \F^k) \to \R_+$ and every map~$\phi \in \Pol_{\leq d}(\F^n, \F^k)$ with~$\rank(\phi) \geq R$, we have
\beqn
\Pr_{I\sim [n]_{\sigma}}\big[ \rank(\phi_{|I}) \geq \kappa\cdot \rank(\phi)\big] \geq 1 - \eps.
\eeqn
\end{theorem}

The proof of this theorem will be given in Section~\ref{sec:polynomials}.
For reasons that will be made clear in that section, this proof will be (at least superficially) quite different from that of the tensor case, Theorem~\ref{thm:random_tensor};
it relies instead on results from analysis of Boolean functions taken together with elementary combinatorial arguments.

\begin{remark}
The quantitative bounds obtained by Theorem~\ref{thm:random_poly} are much weaker than those of Theorem~\ref{thm:random_tensor}.
When $d!$ is invertible in~$\F$, however, many rank functions for polynomial maps~$\phi$ are closely related to rank functions for the symmetric tensor associated to~$\phi$ (see for instance~\cite[Section~3]{GowersW2011}).
In such cases it might be possible to replace an application of Theorem~\ref{thm:random_poly} by Theorem~\ref{thm:random_tensor} so as to obtain good bounds.
\end{remark}

In Section~\ref{sec:discussion} we discuss the relationship between our work and other works present in the literature.
In particular, we will explain the specific problem which led to the study of rank under random coordinate restrictions.
We will also give a simple conditional proof of our main theorems mimicking the case of matrices given in Proposition~\ref{prop:matrixcase}, as long as we assume the existence of a so-called ``linear core'' which would generalize (in a somewhat weak sense) the existence of a full-rank $r\times r$ submatrix inside matrices of rank $r$.
Finally, we propose some open problems to further our understanding of rank functions.

\subsection{Notation}
\label{sec:prelims}

Here we collect some notation that will be used throughout the paper.
Given sets  $I_1, \dots, I_d$, we write $I_{[d]} = I_1\times \cdots\times I_d$.
For a set~$J$ and some~$\sigma \in [0,1]$, we denote by~$\pi_{\sigma}^J$ the product distribution on~$\{0,1\}^J$ where each coordinate is independently set to $1$ with probability~$\sigma$, or to $0$ with probability~$1-\sigma$;
when~$J=[n]$, we write simply~$\pi_{\sigma}^n$.
We write $J_\sigma$ for the probability distribution over subsets of~$J$ where each element is present independently with probability~$\sigma$.
To denote that a random variable~$X$ is distributed according to a distribution~$\mu$, we write $X\sim \mu$.

\section{Tensors}
\label{sec:tensors}

Recall that, for a field $\F$ and integer $d\geq 2$, we denote by $(\F^\infty)^{\otimes d}$ the set of all $d$-tensors of finite support over $\F$.
All our results (including the values of the implied constants) will hold independently of the field considered, so we will always denote it by $\F$ without further comment.

The notions of tensor rank we will consider here are those called \emph{natural rank functions} as defined below:

\begin{definition}[Natural rank]\label{def:nat_rank}
We say that $\rank: (\F^{\infty})^{\otimes d} \to \R_+$ is a natural rank function if $\rank({\bf 0}) = 0$ and it satisfies the following properties:
\begin{enumerate}
    \item Sub-additivity: \\
    $\rank(T+S) \leq \rank(T) + \rank(S)$ for all~$T, S\in (\F^{\infty})^{\otimes d}$.
    \item Monotonicity under restrictions: \\
    $\rank\big(T_{|I_{[d]}}\big) \leq \rank(T)$ for all~$T\in (\F^{\infty})^{\otimes d}$ and all sets~$I_1, \dots, I_d \subset \N$.
    \item Restriction Lipschitz property: \\
    $\rank\big(T_{|J_{[d]}}\big) \leq \rank\big(T_{|I_{[d]}}\big) + \sum_{i=1}^d |J_i \setminus I_i|$
    for all~$T\in (\F^{\infty})^{\otimes d}$ and all sets~$I_1\subseteq J_1, \dots, I_d\subseteq J_d$.
    \end{enumerate}
\end{definition}

In order to motivate this definition, we now discuss several examples (and one non-example) of natural rank functions that have been studied in the literature.
In what follows, given $i\in [d]$ and an element $u\in X_i$, the restriction of~ $T:X_1\times\cdots\times X_d\to \F$ to the set $X_1\times\cdots X_{i-1}\times\{u\}\times X_{i+1}\times\cdots\times X_d$ is referred to as a \emph{slice}.
Note that property~(3) implies that a slice has rank at most~1;
conversely, under the assumption of property~(1), property~(3) is satisfied provided that slices have rank at most~1.

\subsubsection*{Slice rank}
The notion of slice rank was introduced by Tao~\cite{Tao:srank} to give a more symmetric version of Ellenberg and Gijswijt's proof~\cite{EllenbergG2017} of the cap set conjecture, and was later further studied by Sawin and Tao~\cite{TaoSawin}.
It has been used in the study of several extremal combinatorics problems, such as bounding the maximal sizes of tri-colored sum-free sets~\cite{BCCGNSU} and sunflower-free sets~\cite{NaslundS17}, as well as obtaining essentially tight bounds for Green's arithmetic triangle removal lemma~\cite{FoxL17}.

A nonzero $d$-tensor~$T$ (viewed as a multi-linear form) has slice rank~1 if there exist $i\in [d]$,  $R:\F^{X_i}\to \F$ and $S:\prod_{j\in [d]\setminus\{i\}} \F^{X_j}\to \F$ such that~$T$ can be factored as $T = RS$.
In general, the slice rank of~$T$, denoted $\srank(T)$, is then defined as the least $r\in \N$ such that there is a decomposition $T = T_1 + \cdots + T_r$ where each $T_i$ has slice rank~1.
Slice rank is sub-additive since the sum of decompositions of two tensors $S,T$ gives a decomposition of $S+T$.
It is monotone under restrictions since a decomposition of~$T$ induces a decomposition of its restrictions.
The restriction Lipschitz property can easily be verified inductively slice-by-slice using sub-additivity.

\subsubsection*{Partition rank}
The partition rank  was introduced by Naslund~\cite{Naslund2020} as a more general version of the slice rank which allows one to handle problems that require variables to be distinct.
It was first used to provide bounds on the size of subsets of $\F_q^n$ not containing $k$-right corners~\cite{Naslund2020}, as well as an upper bound for the Erd\H{o}s-Ginzburg-Ziv constant of $\F_p^n$~\cite{NaslundEGZ}.

A nonzero $d$-tensor~$T$ is defined to have partition rank~1 if there is a nonempty strict subset $I\subset [d]$ and tensors $R: \prod_{i\in I}\F^{X_i}\to \F$ and $S = \prod_{j\in [d]\setminus I}\F^{X_j}\to \F$ such that~$T$ can be factored as $T = RS$ .
In general, the partition rank of~$T$, denoted $\prank(T)$, is defined as the least $r\in \N$ such that there is a decomposition $T = T_1 + \cdots + T_r$ where each $T_i$ has partition rank~1.
The properties of Definition~\ref{def:nat_rank} for partition rank follow for the same reasons as for slice rank.

\subsubsection*{Analytic rank}
The notion of analytic rank was introduced by Gowers and Wolf~\cite{GowersW2011} to study an arithmetic notion of complexity for linear systems of equations over $\F_p^n$.
It gives a quantitative measure of equidistribution for the values taken by a tensor, and as such are well suited for arguments relying on the dichotomy between structure and randomness.
The analytic rank was further studied by Lovett~\cite{Lovett2019}, and more recently it was used by the first author in a problem concerning a random version of Szemer\'edi's Theorem over finite fields~\cite{Briet:2021}.

The analytic rank is defined only if~$\F$ is finite.
For a non-trivial additive character $\chi:\F\to \C^*$, the bias of a $d$-tensor~$T$ is defined by
\beqn
\bias(T) = \Exp_{x_1\in \F^{X_1},\dots,x_d\in \F^{X_d}}\chi\big(T(x_1,\dots,x_d)\big);
\eeqn
this definition is easily shown to be independent of the character~$\chi$ chosen.
The analytic rank of~$T$ is then defined by
\beqn
\arank(T) = -\log_{|\F|} \bias(T).
\eeqn
Lovett proved that the analytic rank is sub-additive~\cite[Theorem~1.5]{Lovett2019}.
Monotonicity follows from~\cite[Lemma~2.1]{Lovett2019}.
The restriction Lipschitz property now follows from sub-additivity and the fact that a single slice has analytic rank at most~1.

\subsubsection*{Geometric rank}
Motivated by applications in algebraic complexity theory and extremal combinatorics, as well as an open problem posed by Lovett, the geometric rank was introduced by Kopparty, Moshkovitz and Zuiddam~\cite{KoppartyMZ} as a natural extension of analytic rank beyond finite fields.
While its definition is algebraic geometric in nature, it turns out to have deep connections with partition rank and analytic rank, as shown by Cohen and Moshkovitz~\cite{CohenM21, CohenM21b}.

For an algebraically closed field~$\F$, the geometric rank of a $d$-tensor~$T$ is defined as
\beqn
\GR(T) = \codim\{(x_1,\dots,x_{d-1}) \st T(x_1,\dots,x_{d-1},\cdot) = 0\},
\eeqn
where $\codim$ denotes the codimension of an algebraic variety.
If the field $\F$ considered is not algebraically closed, then the geometric rank is naturally defined via the embedding
of $\F$ in its algebraic closure.
Monotonicity and sub-additivity of this rank function follow directly from Lemma~4.2 and Lemma~4.4 of~\cite{KoppartyMZ}, respectively.
The restriction Lipschitz property follows from these properties and the fact that a single slice has geometric rank at most~1.

\subsubsection*{Tensor rank}
It is also instructive to remark on an important non-example, the notion usually known simply as tensor rank, which is an important notion in the context of computational complexity (see for instance~\cite{Landsberg08, Raz13}).
A nonzero $d$-tensor $T$ has tensor rank~1 if it decomposes as $T = \bigotimes_{i\in [d]} u_i$ for some functions $u_i: \F^{X_i} \to \F$.
The tensor rank of a general $d$-tensor~$T$ is then defined as the least $r\in \N$ such that $T = T_1 + \cdots + T_r$, where each $T_i$ has rank~1.

This rank function is not natural according to our Definition~\ref{def:nat_rank} because it fails the restriction Lipschitz property.
Consider for instance the 3-tensor $T\in \F^{[n]\times [n]\times [2]}$ consisting of the identity matrix stacked on top of an all-ones matrix:
\begin{align*}
T(i, j, 1) = 1\big[i=j\big] \,\text{ and }\, T(i, j, 2) = 1 \quad \text{for $i, j \in [n]$.}
\end{align*}
The identity slice has tensor rank~$n$ (which implies that~$T$ has tensor rank at least~$n$), while the all-ones slice has tensor rank 1.
Thus, removing the identity slice $T(\cdot, \cdot, 1)$ reduces the tensor rank of~$T$ by at least~$n-1$, instead of reducing it by at most~1 as required by the restriction Lipschitz condition.
This should not be taken as an indication that our definition is too restrictive, however, as this example also shows that Theorem~\ref{thm:random_tensor} does \emph{not} hold for tensor rank:
a $\sigma$-random restriction of~$T$ has tensor rank at most~1 with probability $1 - \sigma$.
(This simple example can also be generalized in many ways.)

\bigskip

The main result of this section concerns how natural rank functions behave under random coordinate restrictions.
Intuitively, it shows that random restrictions of high-rank tensors will also have high rank with high probability.
For convenience, we repeat below its formal statement as given in the Introduction.

\begin{theorem}[Theorem~\ref{thm:random_tensor} restated]
For every~$d\in \N$ and~$\sigma\in (0,1]$, there exist constants~$C, \kappa > 0$ such that the following holds.
For every natural rank function $\rank: (\F^{\infty})^{\otimes d} \to \R_+$ and every $d$-tensor $T \in \bigotimes_{i=1}^d \F^{[n_i]}$ we have
\beqn
\Pr_{I_1\sim [n_1]_{\sigma}, \dots, I_d\sim [n_d]_{\sigma}}\big[ \rank \big(T_{|I_{[d]}}\big) \geq \kappa\cdot \rank(T)\big] \geq 1 - C e^{-\kappa \rank(T)}.
\eeqn
\end{theorem}

We note that our proof also obtains good quantitative bounds for the parameters in the theorem, namely
\begin{align*}
C = d\sqrt{\frac{3}{2\sigma}} \quad &\text{and} \quad \kappa = \frac{\ln 2}{3} \Big(\frac{\sigma}{4}\Big)^d \quad \text{if } \sigma \in (0, 1/2), \\
C = d\sqrt{2} \quad &\text{and} \quad \kappa = \frac{\ln 2}{2} \Big(\frac{1}{6}\Big)^d \quad \text{if } \sigma \in [1/2, 1].
\end{align*}
Note that one must always have $\kappa \leq \sigma^d$, as can be seen by considering a diagonal tensor~$T$ and the slice rank function~$\srank$, which for diagonal tensors equals the size of their support \cite[Lemma~1]{Tao:srank}.

\subsection{Concentration for monotone sub-additive functions}

The main step in our proof of Theorem~\ref{thm:random_tensor} is a type of concentration inequality for monotone sub-additive functions on the hypercube.

In order to obtain such a result we will consider a notion of \emph{two-point distance} from sets $A\subseteq \bset{n}$, which intuitively measures how many coordinates of some given point $x$ cannot be captured by any two elements of~$A$.

\begin{definition}
Given a point $x\in \bset{n}$ and a set $A \subseteq \bset{n}$, the two-point distance between $x$ and $A$ is
\beqn
h_2(x; A) = \min_{y, z \in A} \big|\big\{i \in [n] \st x_i \neq y_i \And x_i \neq z_i\big\}\big|.
\eeqn
\end{definition}

Note that this ``distance'' $h_2(x; A)$ can be zero even if $x \notin A$;
for instance, $h_2\big(00;\, \{01, 10\}\big) = 0$.
The reason for considering this notion is that one gets much better concentration for $h_2$ than one does for the usual Hamming distance.
This is shown by the following result, which is a special case of an inequality of Talagrand~\cite[Theorem 3.1.1]{Talagrand:1995}:

\begin{theorem}\label{thm:Talagrand}
For any parameter $\sigma\in [0,1]$ and set $A \subseteq \bset{n}$, we have that
\beqn
\Pr_{x\sim \pi_\sigma^n}\big[h_2(x; A) \geq k\big] \leq 2^{-k} \pi_\sigma^n(A)^{-2} \quad \text{for all $k\geq 0$.}
\eeqn
\end{theorem}

The next lemma is the key result needed for proving our main theorem for tensors, and it deals more abstractly with monotone, sub-additive Lipschitz functions on the hypercube.
We endow~$\bset{n}$ with the usual partial order, where $x\leq y$ if the support of~$x$ is contained in the support of~$y$.
A function $f:\bset{n} \to \R$ is monotone if $f(x) \leq f(y)$ whenever $x\leq y$, and it is sub-additive if $f(x+y) \leq f(x) + f(y)$ for all $x,y\in\bset{n}$ with  disjoint supports.
Finally,~$f$ is $1$-Lipschitz if
\beqn
\big|f(x) - f(y)\big| \leq \big|\big\{i \in [n] \st x_i \neq y_i\big\}\big| \quad \text{for all $x, y \in \bset{n}$}.
\eeqn
For a string $x\in \bset{n}$ and set $I\subseteq[n]$, we let~$x_I\in \bset{n}$ be the string that equals~$x$ on the indices in~$I$ and is zero elsewhere.
The lemma that follows and its proof are inspired by an argument of Schechtman~\cite[Corollary~12]{Schechtman2003}.

\begin{lemma}[Concentration inequality] \label{lem:subadditive}
Let $f:\bset{n} \to \R_+$ be a monotone, sub-additive $1$-Lipschitz function such that $f({\bf 0}) =0$ and $f({\bf 1}) = r$.
Then,
\begin{align*}
    \Pr_{x\sim \pi_{\sigma}^n} \big[f(x) \leq \sigma r/4\big] &\leq \sqrt{\frac{3}{2\sigma}}\, 2^{-\sigma r/12} \quad &\text{if $\,0 < \sigma < 1/2$,} \\
    \Pr_{x\sim \pi_{\sigma}^n} \big[f(x) \leq r/6\big] &\leq \sqrt{2}\, 2^{-r/12} &\text{if $\,1/2 \leq \sigma \leq 1$.}
\end{align*}
\end{lemma}

\begin{proof}
We first prove the first inequality above.
Let $\sigma\in (0, 1/2)$, and consider the set
\beqn
A = \big\{ x\in \bset{n} \st f(x)\leq \sigma r/4 \big\}.
\eeqn
By Talagrand's Inequality (Theorem~\ref{thm:Talagrand}) we have
\beqn
\Pr_{x\sim \pi_{\sigma}^n}\big[ h_2(x; A)\geq \sigma r/6\big] \leq 2^{-\sigma r/6} \,\Pr_{x\sim \pi_{\sigma}^n}\big[ f(x)\leq \sigma r/4\big]^{-2},
\eeqn
so to finish the proof it suffices to show that
\beqn
\Pr_{x\sim \pi_{\sigma}^n}\big[ h_2(x; A)\geq \sigma r/6\big] \geq 2\sigma/3.
\eeqn

We claim that
\beq \label{eq:h2inclusion}
\big\{ x\in \bset{n} \st f(x)\geq 2\sigma r/3 \big\} \subseteq \big\{ x\in \bset{n} \st h_2(x; A)\geq \sigma r/6 \big\}.
\eeq
Indeed, if $h_2(x; A) < \sigma r/6$ then there are $y, z\in A$ such that
\beqn
\big|\big\{i \in [n] \st x_i \neq y_i \And x_i \neq z_i\big\}\big| < \sigma r/6.
\eeqn
Denote $I = \big\{i \in [n] \st x_i \neq y_i \And x_i \neq z_i\big\}$, $J = \big\{i \in [n] \st x_i = y_i\big\}$ and $J' = \big\{i \in [n] \st x_i \neq y_i \And x_i = z_i\big\}$;
note that these sets partition~$[n]$, and by assumption $|I| < \sigma r/6$.
We then have
\begin{align*}
    f(x) &\leq f(x_I + x_J) + f(x_{J'}) \\
    &= f(x_I + y_J) + f(z_{J'}) \\
    &\leq |I| + f(y_J) + f(z_{J'}) \\
    &\leq |I| + f(y) + f(z) \\
    &< 2\sigma r/3,
\end{align*}
so $\big\{h_2(x; A) < \sigma r/6\big\} \subseteq \big\{f(x) < 2\sigma r/3\big\}$ and inclusion \eqref{eq:h2inclusion} follows.
It thus suffices to show that $\Pr_{x\sim \pi_{\sigma}^n}\big[ f(x)\geq 2\sigma r/3\big] \geq 2\sigma/3$.

We will next prove that, for any integer $k\geq 1$, we have
\beq \label{eq:coupling}
\Pr_{x\sim \pi_{1/k}^n}\big[ f(x)\geq r/k\big] \geq 1/k.
\eeq
This is done by a simple \emph{coupling argument}, which will be important for us again later on.
Consider a uniformly random ordered $k$-partition $(I_1, \dots, I_k)$ of $[n]$;
thus the $k$ sets $I_i$ are pairwise disjoint and have union $[n]$, with each of the $k^n$ possible such $k$-tuples having the same probability.
Denoting by $\1_I$ the indicator function of set $I$, we have
\beqn
r = f(\1_{[n]}) = f(\1_{I_1} + \dots + \1_{I_k}) \leq \sum_{i=1}^k f(\1_{I_i}),
\eeqn
so $f(\1_{I_i}) \geq r/k$ holds for at least one $i\in [k]$ in \emph{every} ordered partition $(I_1, \dots, I_k)$.
By symmetry, it follows that $\Pr \big[f(\1_{I_1}) \geq r/k \big] \geq 1/k$.
Since the marginal distribution of $\1_{I_1}$ (and every other $\1_{I_i}$) is precisely $\pi_{1/k}^n$, we conclude that
\begin{equation*}
\Pr_{x\sim \pi_{1/k}^n} \big[f(x) \geq r/k\big] = \Pr \big[f(\1_{I_1}) \geq r/k \big] \geq 1/k,
\end{equation*}
as wished.

Now let $k = \lceil 1/\sigma \rceil$.
Since $0 < \sigma < 1/2$, we have $2\sigma/3 \leq 1/k \leq \sigma$, and so by monotonicity of $f$
\begin{align*}
    \Pr_{x\sim \pi_{\sigma}^n}\big[ f(x)\geq 2\sigma r/3\big] &\geq \Pr_{x\sim \pi_{1/k}^n}\big[ f(x)\geq 2\sigma r/3\big] \\
    &\geq \Pr_{x\sim \pi_{1/k}^n}\big[ f(x)\geq r/k\big].
\end{align*}
From inequality~\eqref{eq:coupling} we then conclude that
\beqn
\Pr_{x\sim \pi_{\sigma}^n}\big[ f(x)\geq 2\sigma r/3\big] \geq 1/k \geq 2\sigma/3,
\eeqn
finishing the proof of the first inequality in the statement of the lemma.

The second inequality is proven using the same arguments, now applied to the parameter $\sigma = 1/2$ and the set $A = \big\{ x\in \bset{n} \st f(x)\leq r/6 \big\}$.
By the same argument as above,
\begin{equation*}
    \big\{x\in \bset{n}\st f(x) \geq r/2\big\} \subseteq \big\{x\in \bset{n}\st h_2(x;A) \geq r/6\big\}.
\end{equation*}
Talagrand's inequality and coupling imply that
\begin{align*}
    2^{-r/6}\Pr_{x\sim\pi_{1/2}^n}\big[f(x) \leq r/6\big]^{-2} &\geq \Pr_{x\sim \pi_{1/2}^n}\big[h_2(x;A) \geq r/6\big]\\
    &\geq \Pr_{x\sim\pi_{1/2}^n}\big[f(x) \geq r/2\big]\\
    &\geq \frac{1}{2}.
\end{align*}
Rearranging then gives the result for~$\sigma =  1/2$;
we obtain slightly better bounds since in this case $1/\sigma = 2$ is an integer, and so no rounding errors occur.
By monotonicity of $f$, the same bound continues to hold for all $\sigma > 1/2$.
\end{proof}

\subsection{The Random Restriction Theorem}

It is now a simple matter to use Lemma~\ref{lem:subadditive} to prove the Random Restriction Theorem for tensors.

\begin{proof}[ of Theorem~\ref{thm:random_tensor}]
Let $\rank: (\F^{\infty})^{\otimes d} \to \R$ be a natural rank function, and $T \in \bigotimes_{i=1}^d \F^{[n_i]}$ be a tensor.
We wish to show that
\beqn
\Pr_{I_1\sim [n_1]_{\sigma}, \dots, I_d\sim [n_d]_{\sigma}}\big[ \rank \big(T_{|I_{[d]}}\big) \geq \kappa(\sigma) \cdot \rank(T)\big] \geq 1 - C(\sigma) e^{-\kappa(\sigma) \rank(T)}
\eeqn
for some well-chosen constants $C(\sigma), \kappa(\sigma) > 0$ depending only on the value of $\sigma > 0$ and the order $d$ of the tensor.
We consider here the case where $\sigma < 1/2$, as the case where $\sigma \geq 1/2$ is analogous.\footnote{This second case is also an immediate consequence of first case together with monotonicity of $\rank$ under restrictions.}

Define the function $f_1: \bset{n_1} \to \R_+$ as follows:
for $x \in \bset{n_1}$ with support $J$, $f_1(x)$ is the rank of the subtensor of $T$ whose indices of the first row belong to $J$.
In formula:
\beqn
f_1(\1_{J}) = \rank\big( T_{|J \times \prod_{j=2}^d [n_j]} \big) \quad \text{for all } J\subseteq [n_1].
\eeqn
By the definition of natural rank, $f_1$ is a monotone, sub-additive $1$-Lipschitz function with maximum value $\rank(T)$.
Since $\1_J \sim \pi_{\sigma}^{n_1}$ when $J\sim [n_1]_{\sigma}$, it follows from Lemma~\ref{lem:subadditive} that
\beq \label{eq:f1}
\Pr_{I_1\sim [n_1]_{\sigma}}\bigg[ \rank\big( T_{|I_1 \times \prod_{j=2}^d [n_j]} \big) \geq \frac{\sigma \rank(T)}{4}\bigg] \geq 1 - \sqrt{\frac{3}{2\sigma}}\, 2^{-\frac{\sigma}{12} \rank(T)}.
\eeq

For any fixed $I_1 \subseteq [n_1]$ satisfying $f_1(\1_{I_1}) \geq \sigma \rank(T)/4$, we define the function $f_2: \bset{n_2} \to \R_+$ by
\beqn
f_2(\1_{J}) = \rank\big( T_{|I_1 \times J \times \prod_{j=3}^d [n_j]} \big) \quad \text{for all } J\subseteq [n_2].
\eeqn
This function is again monotone, sub-additive and $1$-Lipschitz, and it has maximum value at least~$\sigma \rank(T)/4$.
By Lemma~\ref{lem:subadditive} we have
\beqn
\Pr_{I_2\sim [n_2]_{\sigma}}\bigg[ \rank\big( T_{|I_1 \times I_2 \times \prod_{j=3}^d [n_j]} \big) \geq \Big(\frac{\sigma}{4}\Big)^2 \rank(T)\bigg] \geq 1 - \sqrt{\frac{3}{2\sigma}}\, 2^{-\frac{\sigma}{12} \frac{\sigma}{4} \rank(T)}.
\eeqn
Since this holds whenever $f_1(\1_{I_1}) \geq \sigma \rank(T)/4$, taking the union bound together with inequality~\eqref{eq:f1} we obtain
\beqn
\Pr_{I_1\sim [n_1]_{\sigma}, I_2\sim [n_2]_{\sigma}}\bigg[ \rank\big( T_{|I_1 \times I_2 \times \prod_{j=3}^d [n_j]} \big) \geq \Big(\frac{\sigma}{4}\Big)^2 \rank(T)\bigg] \geq 1 - 2\sqrt{\frac{3}{2\sigma}}\, 2^{-\frac{\sigma}{12} \frac{\sigma}{4} \rank(T)}.
\eeqn
Proceeding in this same way for each row of the tensor, and always taking the union bound, we eventually conclude that
\beqn
\Pr_{I_1\sim [n_1]_{\sigma}, \dots, I_d\sim [n_d]_{\sigma}}\bigg[ \rank \big(T_{|I_{[d]}}\big) \geq \Big(\frac{\sigma}{4}\Big)^d \rank(T)\bigg] \geq 1 - d\sqrt{\frac{3}{2\sigma}}\, 2^{-\frac{\sigma}{12} (\frac{\sigma}{4})^{d-1} \rank(T)},
\eeqn
which finishes the proof.
\end{proof}

\section{Polynomial maps}
\label{sec:polynomials}

Next we consider the setting of polynomials and higher-dimensional polynomial maps.
Recall that $\Pol_{\leq d}(\F^n, \F^k)$ is the space of all polynomial maps~$\phi: \F^n \to \F^k$ of degree at most~$d$, and $\Pol_{\leq d}(\F^{\infty}, \F^k) = \bigcup_{n\in \N} \Pol_{\leq d}(\F^n, \F^k)$ is the space of all $k$-dimensional polynomial maps over~$\F$ of degree at most~$d$.

\begin{definition}[Natural rank]
We say that $\rank: \Pol_{\leq d}(\F^{\infty}, \F^k) \to \R_+$ is a natural rank function if $\rank({\bf 0}) = 0$ and it satisfies the following properties:
\begin{enumerate}
    \item Symmetry: \\
    $\rank(\phi) = \rank(-\phi)$ for all~$\phi\in \Pol_{\leq d}(\F^{\infty}, \F^k)$.
    \item Sub-additivity: \\
    $\rank(\phi + \gamma) \leq \rank(\phi) + \rank(\gamma)$ for all~$\phi, \gamma\in \Pol_{\leq d}(\F^{\infty}, \F^k)$.
    \item Monotonicity under restrictions: \\
    $\rank(\phi_{|I}) \leq \rank(\phi)$ for all~$\phi\in \Pol_{\leq d}(\F^{\infty}, \F^k)$ and all sets~$I \subset \N$.
    \item Restriction Lipschitz property: \\
    $\rank(\phi_{|I \cup J}) \leq \rank(\phi_{|I}) + |J|$ for all~$\phi\in \Pol_{\leq d}(\F^{\infty}, \F^k)$ and all sets~$I, J \subset \N$.
\end{enumerate}
\end{definition}

Below we discuss some natural rank functions for polynomial maps which have appeared in the literature.
Symmetry holds trivially for each of the ranks discussed.

\subsubsection*{Degree rank}
Motivated by proving a version of the Gowers Inverse Theorem for polynomial phase functions, Green and Tao defined the notion of degree rank for functions over $\F_p^n$ and showed that it is closely related to equidistribution properties of polynomials~\cite{GreenT2009}.
For an integer~$d\geq 1$, the degree-$d$ rank of a function $f:\F^n\to \F$, denoted $\rank_d(f)$, is the least $r\in \N$ such that there exist polynomials $Q_1,\dots,Q_r\in \F[x_1,\dots,x_n]$ of degree at most~$d$ and a function $\Gamma:\F^r\to \F$ such that~$f$ can be decomposed as $f(x) = \Gamma\big(Q_1(x),\dots,Q_r(x)\big)$.
A slightly modified version of this rank function, where one also allows ``nonclassical polynomials'' to enter the decomposition, was instrumental in the proof of the general Gowers Inverse Theorem over~$\F_p^n$ by Tao and Ziegler~\cite{TaoZ12}.

It turns out that, restricted to polynomials of degree at most~$d$, the function $\tfrac{1}{2}\rank_{d-1}$ is a natural rank function.
(The need to divide by~$2$ is a normalization matter which has no important impact.)
Sub-additivity of this rank function follows directly since the sum of decompositions of two polynomials~$P, Q$ gives a decomposition of their sum~$P+Q$.
Monotonicity follows since a decomposition of~$P$ induces a decomposition of any restriction~$P_{|I}$.
The restriction Lipschitz property can be verified inductively using the fact that for a polynomial $P\in \F[x_1,\dots,x_n]$ of degree at most~$d$, we have that $P = P_{|[n-1]} + x_n P'$ for some polynomial~$P'$ of degree at most~$d-1$.
This shows that $\rank_{d-1}(P) \leq \rank_{d-1}(P_{|[n-1]}) + 2$, where the extra~2 comes from the polynomials~$P'$ and~$x_n$.

\subsubsection*{Schmidt rank}
A refinement of the notion of degree-$d$ rank was introduced by Schmidt in~\cite{Schmidt:1985}, which was rediscovered independently on multiple occasions later on.
The same notion, specialized to cubic polynomials, was defined in~\cite{DerksenES:2017} and referred to as q-rank, while in~\cite{KazhdanZ:cohom} it was defined for homogeneous polynomials and referred to simply as rank.
It appeared under the name strength in~\cite{AnanyanH:2020}, it was studied in~\cite{HaramatyS2010} without being explicitly defined and then again in~\cite{Hatami2016}, where it was called strong rank.

The Schmidt rank of a degree-$d$ polynomial $P\in \F[x_1,\dots,x_n]$, which we will denote $\prank(P)$ to reflect its similarity with partition rank for tensors, is the least~$r\in \N$ such that~$P$ has a decomposition of the form $P = Q_1R_1 + \cdots Q_rR_r + Q_{r+1}$
where~$Q_i,R_i$ are non-constant polynomials such that $\deg(Q_i) + \deg(R_i) \leq d$ and $\deg(Q_{r+1}) \leq d-1$.
Sub-additivity and monotonicity follow for the same reasons as for slice and partition rank of tensors.
The argument used for degree-$d$ rank shows the restriction Lipschitz property.

\subsubsection*{Analytic rank}
The next notion of rank is defined more generally for higher-dimensional polynomial maps, and requires the field~$\F$ to be finite.
It was introduced in the context of circuit complexity and error-correcting codes~\cite{BrietBCSN};
see Section~\ref{sec:discussion} for more details.
For a finite field $\F$, the analytic $d$-rank of a map~$\phi\in \Pol_{\leq d}(\F^n; \F^k)$ is defined by
\beqn
    \arank_d(\phi) = -\log_{|\F|} \bigg( \max_{\psi:\F^n\to\F^k,\, \deg(\psi)<d} \Pr_{x\in \F^n} \big[\phi(x) = \psi(x)\big] \bigg).
\eeqn
Sub-additivity, monotonicity and the restriction Lipschitz property for this rank function were proven in \cite[Section~5]{BrietBCSN}.

Finally, we note that there is also another notion of analytic rank specific to polynomials, which equals the tensor analytic rank of their associated homogeneous multilinear form.
It was introduced together with the notion of tensor analytic rank by Gowers and Wolf~\cite{GowersW2011}, and can be equivalently defined for polynomials $P$ of degree~$d$ by
$$ -\log_{|\F|} \|\chi(P)\|_{U^d}^{2^d},$$
where~$\chi$ is a nontrivial additive character of~$\F$ and $\|\cdot\|_{U^d}$ is the Gowers uniformity norm of order~$d$.
Sub-additivity and monotonicity under restrictions follow from the corresponding properties for tensor analytic rank. 
Upon normalization by a factor~$1/d$, the restriction Lipschitz property follows from the same property for the analytic rank of tensors.
This notion of rank was also important in Tao and Ziegler's proof of the Gowers Inverse Theorem over~$\F_p^n$~\cite{TaoZ12}.

\bigskip

Our main result shows that random restrictions of a high-rank polynomial map will also have high rank with high probability.
We recall below its formal statement as given in the Introduction.

\begin{theorem}[Theorem~\ref{thm:random_poly} restated]
For every~$d\in \N$ and~$\sigma,\eps\in (0,1]$, there exist constants~$\kappa = \kappa(d, \sigma)>0$ and~$R = R(d,\sigma,\eps)\in\N$ such that the following holds.
For every natural rank function~$\rank: \Pol_{\leq d}(\F^{\infty}, \F^k) \to \R$ and every map~$\phi \in \Pol_{\leq d}(\F^n, \F^k)$ with~$\rank(\phi) \geq R$, we have
\beqn
\Pr_{I\sim [n]_{\sigma}}\big[ \rank(\phi_{|I}) \geq \kappa\cdot \rank(\phi)\big] \geq 1 - \eps.
\eeqn
\end{theorem}

The proof of this theorem proceeds differently from the proof of Theorem~\ref{thm:random_tensor}, which gives the analogous result for tensors.
The reason for this is that the natural choice of function $f:\bset{n}\to \R_+$ to which one might apply the concentration inequality in Lemma~\ref{lem:subadditive}, given by $f(\one_I) = \rank(\phi_{|I})$ (for some rank function~$\rank$), might fail to be sub-additive as a function on the boolean hypercube while~$\rank$ is sub-additive as a natural rank function.
For instance, consider (in dimension~$k=1$) a $2n$-variate polynomial $p(x_1,\dots,x_n,y_1,\dots,y_n)$ such that each monomial contains both an~$x$ variable and a~$y$ variable.
Then the rank of~$p$ restricted only to its~$x$ variables is zero, and similarly for the restriction to the~$y$ variables, while~$\rank(p)$ can be of order~$\Theta(n)$.
This shows that~$\rank(p_{|I \cup J})$ can be much larger than $\rank(p_{|I}) + \rank(p_{|J})$, and the argument used in the tensor case breaks down.

\subsection{The proof in expectation}

The first step in our proof of Theorem~\ref{thm:random_poly} is to show that the desired result is true \emph{in expectation}, rather than with high probability.
More precisely, we wish to show an inequality of the form
\beqn
\Exp_{I\sim [n]_{\sigma}} \rank(\phi_{|I}) \geq c(\sigma, d) \rank(\phi)
\eeqn
which is valid for all degree-$d$ maps $\phi\in \Pol_{\leq d}(\F^n, \F^k)$ and all natural rank functions.

Towards proving the inequality above by induction on the degree, it will be convenient to generalize it to polynomial maps over \emph{commutative rings} rather than only fields.
The reason for doing so is that, by considering some subset of variables $\{x_a: a\in A\}$ to be constants, we might be able to decrease the degree of the associated polynomial map over the remaining variables $\{x_b: b\notin A\}$.

More formally, given a polynomial map $\phi\in \Pol_{\leq d}(\F^n, \F^k)$ and two disjoint subsets $A, B \subset [n]$, denote by $\phi[A, B]$ the sum of the monomials in~$\phi$ involving at least one variable from each~$A$ and~$B$.
Note that $\phi[A, B]$ can be regarded as a polynomial map in three different ways:
either as having coefficients in~$\F$ and variables in $\{x_i: i\in [n]\}$;
or having coefficients in the ring $\F[x_b: b\in B]$ and variables in $\{x_a: a\in A\}$;
or having coefficients in $\F[x_a: a\in A]$ and variables in $\{x_b: b\in B\}$.
The advantage of these last two representations is that the degree of $\phi[A, B]$ \emph{decreases}, since every monomial will contain formal variables which are now elements of the ring (and are thus regarded as constants).

For a commutative ring~$R$, denote by $R[x_1, \dots, x_n]_{\leq d}$ the set of formal polynomials in variables $x_1, \dots, x_n$ with coefficients over~$R$ and degree at most~$d$.
For a bipartition $[n] = A\uplus B$, define the polynomial ring $R_A = R[x_a\st a\in A]$ and consider the natural identification of $R[x_1,\dots,x_n]$ and $R_A[x_b\st b\in B]$.
Under this identification, a function $\rank: (R[x_1, \dots, x_n])^k \to \R_+$ induces a function $\rank:(R_A[x_b\st b\in B])^k\to \R_+$ in a way that preserves symmetry, sub-additivity and monotonicity under restrictions.

\begin{lemma}[Linear expectation] \label{lem:expectation}
For every~$\sigma\in (0, 1]$ and every integer~$d \geq 1$ there exists a constant~$c(\sigma, d) > 0$ such that the following holds.
Let~$n,k$ be positive integers,~$R$ be a commutative ring and $\rank: (R[x_1, \dots, x_n]_{\leq d})^k \to \R_+$ be symmetric, sub-additive and monotone under restrictions.
Then, for any map $\phi\in (R[x_1, \dots, x_n]_{\leq d})^k$ we have
\beqn
\Exp_{J\sim [n]_{\sigma}} \rank(\phi_{|J}) \geq c(\sigma, d) \rank(\phi).
\eeqn
\end{lemma}

\begin{proof}
For simplicity, we will prove the result in the special case where $\sigma=1/2$.
Since $\phi_{|I\cap J} = (\phi_{|I})_{|J}$ and $I\cap J \sim [n]_{1/2^{j+1}}$ when $I\sim [n]_{1/2}$, $J\sim [n]_{1/2^j}$, the general case easily follows from this special case by taking constant
\beqn
c(\sigma, d) = c(1/2, d)^{\lceil\log(1/\sigma)\rceil}
\eeqn
and using monotonicity under restrictions.

Denote $c(d) := c(1/2, d)$ for ease of notation.
The proof will proceed by induction on the degree~$d$, with the base case where $d=0$ holding trivially with $c(0) = 1$.
For the inductive step, let
\beqn
c(d) = \frac{c(d-1)^2}{20}
\eeqn
and consider two different possibilities.

First, suppose that for every bipartition $[n] = A\uplus B$ we have that
\beqn
\max \big\{\rank(\phi_{|A}),\, \rank(\phi_{|B})\big\} \geq 2c(d) \rank(\phi).
\eeqn
Using a similar coupling argument as was used in the proof of Lemma~\ref{lem:subadditive}, we get that
\begin{align*} \Pr_{I\sim[n]_{1/2}}\big[\rank\big(\phi_{|I}\big) \geq 2c(d)\rank(\phi)\big] \geq \frac{1}{2},
\end{align*}
which implies
\begin{equation*}
    \Exp_{I\sim [n]_{1/2}}\rank\big(\phi_{|I}\big) \geq c(d)\rank(\phi)
\end{equation*}
as desired.

The second possibility is that there exists a bipartition $[n] = A\uplus B$ such that
\beq \label{eq:second_case}
\rank(\phi_{|A}) < 2c(d) \rank(\phi) \quad \text{and} \quad \rank(\phi_{|B}) < 2c(d) \rank(\phi).
\eeq
Fix such a bipartition for the rest of the proof.
For any given subsets $I\subseteq A$, $J\subseteq B$, let $\phi[I, J]$ be the sum of the monomials in~$\phi$ involving at least one variable from each~$I$ and~$J$.
It is easy to see that
\beqn
\phi[I, J] = \phi_{|I\cup J} - \phi_{|I} - \phi_{|J} + \phi(0).
\eeqn
Define the commutative rings $R_A := R[x_a: a\in A]$ and $R_B := R[x_b: b\in B]$, and note that $\phi[I, J]$ can be seen both as an element of $(R_A[x_b: b\in B])^k$ and as an element of $(R_B[x_a: a\in A])^k$.
Moreover, the degree of $\phi[I, J]$ is at most $d-1$ in each of these representations.

From the identity $\phi = \phi_{|A} + \phi_{|B} - \phi(0) + \phi[A, B]$ and sub-additivity, it follows that
\beqn
\rank(\phi[A, B]) \geq \rank(\phi) - \rank(\phi_{|A}) - \rank(\phi_{|B}) - \rank(-\phi(0)).
\eeqn
Since $\rank(-\phi(0)) = \rank(\phi(0)) = \rank(\phi_{|\emptyset})$, we conclude from monotonicity under restrictions and our assumption~\eqref{eq:second_case} that
\beq\label{eq:phiAB}
\rank(\phi[A, B]) > \rank(\phi) - 6 c(d) \rank(\phi) > \rank(\phi)/2.
\eeq
In an analogous way, we obtain the inequalities
\begin{align} \label{eq:phi_IJ}
    \rank(\phi_{|I\cup J}) &\geq \rank(\phi[I, J]) - \rank(\phi_{|I}) - \rank(\phi_{|J}) - \rank(\phi(0)) \\
    &> \rank(\phi[I, J]) - 6 c(d) \rank(\phi) \notag
\end{align}
valid for all $I\subseteq A$, $J\subseteq B$.

Applying the inductive hypothesis to $\phi[A, B] \in (R_B[x_a: a\in A]_{\leq d-1})^k$ we obtain
\beqn
\Exp_{I\sim A_{1/2}} \rank(\phi[I, B]) = \Exp_{I\sim A_{1/2}} \rank(\phi[A, B]_{|I}) \geq c(d-1) \rank(\phi[A, B]).
\eeqn
Moreover, for any fixed $I\subseteq A$, the inductive hypothesis applied to the map $\phi[I, B] \in (R_A[x_b: b\in B]_{\leq d-1})^k$ gives
\beqn
\Exp_{J\sim B_{1/2}} \rank(\phi[I, J]) = \Exp_{J\sim B_{1/2}} \rank(\phi[I, B]_{|J}) \geq c(d-1) \rank(\phi[I, B]).
\eeqn
Combining the two inequalities above we conclude that
\beqn
\Exp_{I\sim A_{1/2},\, J\sim B_{1/2}} \rank(\phi[I, J]) \geq c(d-1)^2 \rank(\phi[A, B]),
\eeqn
and thus by~\eqref{eq:phiAB} we have
$\Exp_{I\sim A_{1/2},\, J\sim B_{1/2}} \rank(\phi[I, J]) > c(d-1)^2 \rank(\phi)/2$.
Together with equation~$\eqref{eq:phi_IJ}$, we conclude that
\begin{align*}
    \Exp_{I\sim A_{1/2},\, J\sim B_{1/2}}\rank(\phi_{|I\cup J})
    &> \Exp_{I\sim A_{1/2},\, J\sim B_{1/2}} \rank(\phi[I, J]) - 6c(d) \rank(\phi) \\
    &> c(d-1)^2 \rank(\phi)/2 - 6c(d) \rank(\phi) \\
    &> c(d) \rank(\phi),
\end{align*}
which is precisely what we wanted to prove.
\end{proof}

\subsection{Monotone functions on the hypercube and boosting}

Our next result is a lemma which allows us to boost the probability of some events (such as having high rank under random restrictions) from~$\eps$ to~$1-\eps$ by paying a relatively small price.

It will again be convenient to take a more abstract approach and deal with Boolean functions rather than restrictions of polynomials.
For a Boolean function~$g: \{0,1\}^n \to \{0,1\}$, the total influence of~$g$ under distribution~$\pi_{\sigma}^n$ is given by
\beqn
\mathbf{I}^{(\sigma)}(g)
=
\sum_{i=1}^n \Exp_{x\sim\pi_\sigma^n}\big|g(x) - g(x^i)\big|,
\eeqn
where~$x^i$ differs from~$x$ only in the~$i$th coordinate.
Denote by~$\mu_{\sigma}[g] := \Exp_{x\sim \pi_{\sigma}^n} \big[g(x)\big]$ its expectation.

\begin{lemma}[Boosting lemma] \label{lem:boost}
For every~$\sigma>0$ there is a constant~$C_{\sigma}>0$ such that the following holds.
Let~$n\in \N$, $f: \{0,1\}^n \to \R_+$ be a monotone~$1$-Lipschitz function and~$\eps\in (0, 1/2]$.
If~$r\geq 1$ satisfies
\beqn
\Pr_{x\sim \pi_{\sigma}^n} \big[f(x) \geq r\big] > \eps,
\eeqn
then~$\Pr_{x\sim \pi_{1.01\sigma}^n} \big[f(x) \geq r - C_{\sigma}^{1/\eps^2}\big] > 1-\eps$.
\end{lemma}

\begin{proof}
Consider the Boolean function~$g(x) := 1 \big[f(x) \geq r\big]$.
Note that there exists~$q \in [\sigma,\, 1.01\sigma]$ such that~$\mathbf{I}^{(q)}(g) \leq 100/\sigma$, as otherwise by the Margulis-Russo formula we would have
\beqn
\mu_{1.01\sigma}[g] = \mu_{\sigma}[g] + \int_{\sigma}^{1.01\sigma} \frac{d\mu_q[g]}{dq} dq \geq \int_{\sigma}^{1.01\sigma} \mathbf{I}^{(q)}(g) dq > 1.
\eeqn

Let~$\gamma \in (0, \eps/2)$ be a constant to be chosen later.
By Friedgut's Junta Theorem~\cite{Friedgut1998}, there is a~$C(q)^{100/\gamma \sigma}$-junta~$h: \{0,1\}^n \to \{0,1\}$ such that
\beqn
\Pr_{x\sim \pi_q^n} \big[g(x) \neq h(x)\big] < \gamma.
\eeqn
Let~$J \subseteq [n]$, $|J| \leq C(q)^{100/\gamma \sigma}$, be the set of variables on which~$h$ depends.
Then
\begin{align*}
    \mu_q[h] \geq \mu_q[g] - \gamma \geq \mu_{\sigma}[g] - \gamma > \eps/2,
\end{align*}
which implies~$\Pr_{z\sim \pi_q^J} \big[h(z) = 1\big] > \eps/2$.
Since
\beqn
\Exp_{z\sim \pi_q^J} \Pr_{y\sim \pi_q^{J^c}} \big[g(y,z) \neq h(z)\big] < \gamma,
\eeqn
it follows that
\beqn
\Pr_{z\sim \pi_q^J} \big[\Pr_{y\sim \pi_q^{J^c}} \big[g(y,z) \neq h(z)\big] \geq \eps\big] < \gamma/\eps.
\eeqn
Taking~$\gamma = \eps^2/2$, we conclude there exists~$z\in \{0,1\}^J$ such that~$h(z)=1$ and
\beqn
    \Pr_{y\sim \pi_q^{J^c}} \big[g(y,z) = 0\big] = \Pr_{y\sim \pi_q^{J^c}} \big[g(y,z) \neq h(z)\big] < \eps.
\eeqn

Since~$f$ is~$1$-Lipschitz by assumption, we have
\beqn
    f(y,z) \geq r \implies f(y,z') \geq r - |J| \quad \forall z'\in \{0,1\}^J,
\eeqn
and thus by monotonicity
\begin{align*}
    \Pr_{x\sim \pi_{1.01\sigma}^n} \big[f(x) \geq r - |J|\big]
    &\geq \Pr_{x\sim \pi_q^n} \big[f(x) \geq r - |J|\big] \\
    &= \Exp_{z'\sim \pi_q^J} \Pr_{y\sim \pi_q^{J^c}} \big[f(y,z') \geq r - |J|\big] \\
    &\geq \max_{z\in \{0,1\}^J} \Pr_{y\sim \pi_q^{J^c}} \big[f(y,z) \geq r\big] \\
    &> 1-\eps.
\end{align*}
This is precisely what we wanted to prove, with constant
$$C_{\sigma} = \sup_{q\in [\sigma,\, 1.01\sigma]} C(q)^{200/\sigma}.$$
Note that the above supremum is finite:
in \cite[Section~10.3]{ODonnell:2014} it is shown that we can take
\beqn
C(q) = \max\big\{ q^{-c},\, (1-q)^{-c}\big\}
\eeqn
for some absolute constant $c>0$.
\end{proof}

\subsection{The Random Restriction Theorem}

Our main result, Theorem~\ref{thm:random_poly}, follows easily from the lemmas given above.
Recall that we wish to show that
\beqn
\Pr_{J\sim [n]_{\sigma}}\big[ \rank(\phi_{|J}) \geq \kappa(d, \sigma)\cdot \rank(\phi)\big] \geq 1 - \eps
\eeqn
whenever~$\rank(\phi) \geq R(d, \sigma, \eps)$, for some well-chosen constants~$R(d, \sigma, \eps)$ and $\kappa(d, \sigma) > 0$.

\begin{proof}[ of Theorem~\ref{thm:random_poly}]
Applying Lemma \ref{lem:expectation} with~$\sigma$ substituted by~$0.9\sigma$, we obtain
\beqn
\Exp_{J\sim [n]_{0.9\sigma}} \rank(\phi_{|J}) \geq c(0.9\sigma, d) \rank(\phi).
\eeqn
Since~$\rank(\phi_{|J}) \leq \rank(\phi)$ for all subsets~$J \subseteq [n]$, we conclude that
\beqn
\Pr_{J\sim [n]_{0.9\sigma}} \bigg[ \rank(\phi_{|J}) \geq \frac{c(0.9\sigma, d)}{2} \rank(\phi) \bigg] \geq \frac{c(0.9\sigma, d)}{2}.
\eeqn

By possibly decreasing~$\eps$ a little, we may assume that~$\eps < c(0.9\sigma, d)/2$.
Denote
\beqn
\kappa(d, \sigma) = \frac{c(0.9\sigma, d)}{4} \quad \text{and} \quad R(d, \sigma, \eps) = \frac{4 C_{0.9\sigma}^{1/\eps^2}}{c(0.9\sigma, d)},
\eeqn
where~$C_{0.9\sigma} > 0$ is the constant guaranteed by the boosting lemma, Lemma \ref{lem:boost}.
Applying the boosting lemma to the function~$J \mapsto \rank(\phi_{|J})$ with~$r = c(0.9\sigma, d) \rank(\phi)/2$, we get
\begin{align*}
    \Pr_{J\sim [n]_{0.9\sigma}} \big[ \rank(\phi_{|J}) \geq r \big] &\geq c(0.9\sigma, d)/2 > \eps \\
    &\implies \Pr_{J\sim [n]_{\sigma}} \big[ \rank(\phi_{|J}) \geq r - C_{0.9\sigma}^{1/\eps^2} \big] > 1-\eps.
\end{align*}
Since~$c(0.9\sigma, d) \rank(\phi)/2 - C_{0.9\sigma}^{1/\eps^2} \geq \kappa(\sigma, d) \rank(\phi)$ whenever $\rank(\phi) \geq R(d, \sigma, \eps)$, the result follows.
\end{proof}

\section{Discussion and open problems}
\label{sec:discussion}

The motivation for this work stems from a problem from theoretical computer science which concerns decoding corrupted error-correcting codes (ECCs) with NC$^0[\oplus]$ circuits~\cite{BrietBCSN}.
An ECC is a map $E:\F_2^k\to \F_2^n$ with the property that the points in its image (codewords) are well-separated in Hamming distance, enabling retrieval of encoded messages provided codewords are not too badly corrupted during transmission or storage.
A standard noise model for corruption is the \emph{symmetric channel}:
for a parameter $\sigma\in [0,1]$ and given an element $y\in \F_2^n$, this channel samples a set $I\sim [n]_{\sigma}$ and, for each $i\in I$, replaces the coordinate~$y_i$ with a uniformly distributed element over~$\F_2$.
The problem is to determine whether NC$^0[\oplus]$ circuits are capable of correctly decoding corrupted codewords with good probability over this noise distribution.

It turns out that the mappings such circuits can effect are precisely constant-degree polynomial maps $\phi:\F_2^n\to\F_2^k$.
One of the main results in~\cite{BrietBCSN} shows that the fraction of messages such maps can correctly decode with non-negligible probability (over the symmetric channel noise model) tends to zero as the size~$k$ of the messages grows.
The proof of this result uses a structure-versus-randomness strategy to analyze the decoding capability of~$\phi$ based on the value $\arank_d(\phi)$ for $d = \deg(\phi)$ (the proof uses no assumptions on the specific ECC).
The key to analyzing the case when this rank is high~-- the pseudorandom setting~-- is to understand how this rank behaves under random restrictions originating from the symmetric channel noise;
this is what motivated Theorem~\ref{thm:random_poly}.
While in this application the degree of~$\phi$ may exceed the field size, a similar result but with quantitatively stronger bounds can be obtained in the high-characteristic setting, where $\charac(\F) > \deg(\phi)$, by using Theorem~\ref{thm:random_tensor}.

\medskip
As already remarked in the Introduction, our proof of Theorem~\ref{thm:random_tensor} for higher-order tensors proceeds quite differently from the matrix case (Proposition~\ref{prop:matrixcase}).
The reason for this is that an analogous proof would require the existence of a high-rank sub-tensor;
to explain this more precisely, we introduce the following definition.

\begin{definition}[Core property]
\label{def:core}
Let $A, B: \R_+ \to \R_+$ be unbounded increasing functions, and let $\rank: (\F^{\infty})^{\otimes d} \to \R_+$.
We say that~$\rank$ satisfies the \emph{$(A,B)$-core property} if, for every $d$-tensor $T\in (\F^{\infty})^{\otimes d}$ of high enough rank~$\rank(T)$, there exist sets $J_1, \dots, J_d \subset \N$ of size at most $A(\rank(T))$ such that $\rank(T_{|J_{[d]}}) \geq B(\rank(T))$.
We say that~$\rank$ satisfies the \emph{linear core property} if it satisfies the $(A,B)$-core property for linear functions $A(r) = Lr$ and $B(r) = \beta r$, with $L,\beta > 0$.
\end{definition}

Draisma~\cite{Draisma:2019} proved that slice rank has a core-like property, up to local linear transformation and provided~$\F$ is infinite.
In particular his result shows that, for every $d,r\in \N$ and every $d$-tensor $T\in \bigotimes_{i=1}^d \F^{X_i}$ of slice rank at least~$r$, there are linear maps $\varphi_i:\F^{X_i}\to \F^n$ such that the tensor $(\varphi_1\otimes\cdots\otimes\varphi_d)T$ also has slice rank at least~$r$,
where~$n = n(d,r)$ is a constant depending only on~$d$ and~$r$.

More in line with Definition~\ref{def:core}, Karam~\cite{Karam2022} recently proved that several rank functions for tensors defined in terms of decompositions, including the slice rank, partition rank and tensor rank, satisfy the $(A,B)$-core property for some functions $A$ and $B$.
For partition rank over finite fields, for instance, he obtains explicit functions $A(r) = \exp(O_{d,\F}(r))$ and $B(r) = \Omega_d(r/(\log r)^d)$;
for the slice rank of 3-tensors he shows that one can take $A(r)=O(r)$ and $B(r) = \Omega(r^{1/3})$;
and for tensor rank, he shows the ``perfect'' linear core property $A(r) = B(r) = r$.
Karam conjectures, moreover, that all these rank functions in fact satisfy the linear core property~\cite[Conjecture~13.1]{Karam2022}.
In this case, a similar argument to the one we used for matrices in the introduction allows us to easily deduce a Random Restriction Theorem:

\begin{theorem}\label{thm:restriction_core}
Suppose $\rank: (\F^{\infty})^{\otimes d} \to \R_+$ satisfies the linear core property, monotonicity under restrictions and the restriction Lipschitz property.
Then for every~$\sigma\in (0,1]$ there exist constants~$C, \kappa>0$ such that, for every $d$-tensor $T \in \bigotimes_{i=1}^d \F^{[n_i]}$, we have
\beqn
\Pr_{I_1\sim [n_1]_{\sigma}, \dots, I_d\sim [n_d]_{\sigma}}\big[ \rank \big(T_{|I_{[d]}}\big) \geq \kappa\cdot \rank(T)\big] \geq 1 - C e^{-\kappa \rank(T)}.
\eeqn
\end{theorem}

\begin{proof}
Let $L,\beta > 0$ be the constants in the linear core property of $\rank$, and denote $\lambda = \beta/(3dL)$.
For a given $d$-tensor $T\in \bigotimes_{i=1}^d \F^{[n_i]}$, fix sets $J_1 \subseteq [n_1], \dots, J_d \subseteq [n_d]$ of size $L \rank(T)$ such that $\rank(T_{|J_{[d]}}) \geq \beta \rank(T)$.

Let $I_1\sim (J_1)_{1-\lambda}, \dots, I_d\sim (J_d)_{1-\lambda}$ be random sets and consider the random event
\beqn
E = \big\{|I_i| \geq (1-2\lambda) |J_i|\, \text{ for all } 1\leq i \leq d\big\}.
\eeqn
Whenever $E$ holds, by the restriction Lipschitz property we have
\begin{align*}
    \rank(T_{|I_{[d]}})
    &\geq \rank(T_{|J_{[d]}}) - \sum_{i=1}^d |J_i \setminus I_i| \\
    &\geq \beta \rank(T) - d \cdot 2\lambda L \rank(T) \\
    &= \beta \rank(T)/3.
\end{align*}
By the Chernoff bound and union bound, the probability of $E$ is at least $1- d\, e^{-(\beta/12d) \rank(T)}$;
it then follows from monotonicity that
\beqn
\Pr_{I_1\sim [n_1]_{1-\lambda}, \dots, I_d\sim [n_d]_{1-\lambda}} \bigg[\rank(T_{|I_{[d]}}) \geq \frac{\beta}{3} \rank(T)\bigg] \geq 1- d\, e^{-\frac{\beta}{12d} \rank(T)}.
\eeqn

Now we apply the same argument to the (random) tensor $\Tilde{T} = T_{|I_{[d]}}$, and union-bound with the event $\big\{ \rank(T_{|I_{[d]}}) \geq \beta\rank(T)/3 \big\}$.
Since $(I_i)_{1-\lambda}$ is distributed as $[n_i]_{(1-\lambda)^2}$ when $I_i\sim [n_i]_{1-\lambda}$, we conclude that
\beqn
\Pr_{I_1\sim [n_1]_{(1-\lambda)^2}, \dots, I_d\sim [n_d]_{(1-\lambda)^2}} \bigg[\rank(T_{|I_{[d]}}) \geq \Big(\frac{\beta}{3}\Big)^2 \rank(T)\bigg] \geq 1- 2d\, e^{-\frac{\beta}{12d} \frac{\beta}{3} \rank(T)}.
\eeqn
In general, applying this same argument recursively $t$ times in total, we get
\beqn
\Pr_{I_1\sim [n_1]_{(1-\lambda)^t}, \dots, I_d\sim [n_d]_{(1-\lambda)^t}} \Big[\rank(T_{|I_{[d]}}) \geq \Big(\frac{\beta}{3}\Big)^t \rank(T)\Big] \geq 1- td\, e^{-\frac{\beta}{12d} (\frac{\beta}{3})^{t-1} \rank(T)}.
\eeqn
The theorem now follows by taking $t$ to be the smallest integer for which $(1-\lambda)^t \leq \sigma$, and using monotonicity under restrictions.
\end{proof}

We quickly remark on another interesting connection between our results and recent work on high-rank maps by Gowers and Karam~\cite{GowersK2022}.
These authors studied equidistribution properties of polynomials and multilinear forms on $\F_p^n$ when the variables are restricted to subsets of their domain;
this setting is quite similar to what motivated our studies, as explained in the beginning of this section.
A crucial step in their arguments was a result (Proposition~3.5 in~\cite{GowersK2022}) showing that the values taken by multilinear forms of high partition rank must be close to uniformly distributed under a wide range of non-uniform input distributions.
Under the well-known conjecture (within additive combinatorics) that the partition rank and analytic rank of tensors are equivalent up to a multiplicative constant, their result would straightforwardly imply the conclusion of our Random Restriction Theorem (Theorem~\ref{thm:random_tensor}) when restricted to either the partition rank or the analytic rank of tensors, although making use of very different arguments.

\bigskip

Analogous to Definition~\ref{def:core}, it also makes sense to define a core property for polynomial maps (which include single polynomials by setting $k = 1$).

\begin{definition}
\label{def:core_poly}
Let $A, B: \R_+ \to \R_+$ be unbounded increasing functions, and let $\rank: \Pol_{\leq d}(\F^n;\F^k) \to \R_+$.
We say that~$\rank$ satisfies the \emph{$(A,B)$-core property} if, for every polynomial map $\phi\in \Pol_{\leq d}(\F^n;\F^k)$ of high enough rank~$\rank(\phi)$, there exists a set $I\subseteq [n]$ of size at most $A(\rank(\phi))$ such that $\rank(\phi_{|I}) \geq B(\rank(\phi))$.
We say that~$\rank$ satisfies the \emph{linear core property} if it satisfies the $(A,B)$-core property for linear functions $A(r) = Lr$ and $B(r) = \beta r$, with $L,\beta > 0$.
\end{definition}

A property comparable to a core property for polynomial maps was proved by Kazhdan and Ziegler~\cite{KazhdanZ:2020}.
For simplicity, we state it here only for polynomials.
They showed that a polynomial $P\in \F[x_1,\dots,x_n]$ with high Schmidt rank and $\deg(P) < \charac(\F)$ is universal in the following sense: 
If $\prank(P) \geq r$, then for any $Q\in \F[x_1,\dots,x_m]$ of degree $\deg(P)$ there is an affine map $\phi:\F^m\to\F^n$ such that $Q = P\circ \phi$, provided~$r$ is large enough in terms of~$m$.
Taking~$Q$ to be a polynomial of maximal Schmidt rank~$m$ shows that, up to an affine transformation,~$P$ restricts to an $m$-variate polynomial of maximal Schmidt rank.

It would be of interest to know which notions of rank for polynomial maps have a core property, and especially if any of the rank functions discussed here have the linear core property (see~\cite{BlatterDR:2022} for recent progress on this).
The same proof as given for Theorem~\ref{thm:restriction_core} shows that a linear core property implies a Random Restriction Theorem for polynomial maps that is quantitatively stronger than Theorem~\ref{thm:random_poly}.

\section*{Acknowledgements}

The authors would like to thank Victor Souza for bringing Karam's paper~\cite{Karam2022} to their attention, and Jan Draisma and Tamar Ziegler for a number of helpful pointers to the literature.
We also thank Thomas Karam for helpful comments on a preliminary version of this manuscript.
We are indebted to the anonymous reviewer for simplifying our original proof of Lemma~\ref{lem:expectation}, and for suggestions which improved the presentation of the paper.





\begin{dajauthors}
\begin{authorinfo}[jb]
  Jop Bri\"{e}t\\
  CWI \& QuSoft\\ 
  Amsterdam, The Netherlands\\
  j.briet\imageat{}cwi\imagedot{}nl\\
  \url{https://www.cwi.nl/~jop}
\end{authorinfo}
\begin{authorinfo}[dcs]
  Davi Castro-Silva\\
  CWI \& QuSoft\\
  Amsterdam, The Netherlands\\
  davisilva15\imageat{}gmail\imagedot{}com\\
  \url{https://sites.google.com/view/davicastrosilva}
\end{authorinfo}
\end{dajauthors}

\end{document}